\documentclass[a4paper]{amsart}

\usepackage{amssymb}              
\usepackage{lmodern,bm}              
\usepackage{longtable}
\usepackage{tikz}
\usepackage[all]{xy}
\usepackage{dsfont}
\usepackage{pdflscape}
\usepackage{bigdelim}
\usepackage{multirow}
\usepackage{cancel}
\usepackage{array}
\usepackage{tikz-cd}

\CompileMatrices

\newtheorem{theorem}{Theorem}[section]

\newtheorem{lemma}[theorem]{Lemma}
\newtheorem{proposition}[theorem]{Proposition}
\newtheorem{corollary}[theorem]{Corollary}

\theoremstyle{definition}

\newtheorem{definition}[theorem]{Definition}
\newtheorem{construction}[theorem]{Construction}

\newtheorem{example}[theorem]{Example}
\newtheorem{remark}[theorem]{Remark}

\numberwithin{equation}{theorem}

\newcommand{\git}{\mathbin{
  \mathchoice{/\mkern-6mu/}
    {/\mkern-6mu/}
    {/\mkern-5mu/}
    {/\mkern-5mu/}}}

\def\vector2#1#2{\left(\begin{array}{c} #1 \\ #2 \end{array}\right)}
\def\Cl{{\rm Cl}}

\def\CC{{\mathbb C}}

\def\TT{{\mathbb T}}
\def\ZZ{{\mathbb Z}}

\def\QQ{{\mathbb Q}}

\def\OOO{{\mathcal O}}
\def\RRR{{\mathcal R}}

\def\Mat{{\rm Mat}}

\def\div{{\rm div}}
\def\quot{/\!\!/}

\def\im{{\rm im}}

\def\bangle#1{{\langle #1 \rangle}}

\def\Spec{{\rm Spec}}

\def\im{{\rm im}}
\def\V{{\rm V}}

\usepackage{hyperref}

\title[Divisor class groups of rational trinomial varieties]%
{Divisor class groups of rational trinomial varieties}

\author[Milena~Wrobel]{Milena Wrobel} 

\address{Max-Planck-Institut f\"ur Mathematik in den Naturwissenschaften, Inselstra\ss e~22, 04103 Leipzig, Germany}
\email{wrobel@mis.mpg.de}

\begin{document}

\begin{abstract}
We give an explicit description of
the divisor class groups of rational 
trinomial varieties.
As an application, we relate the
iteration of Cox rings of any 
rational variety 
with torus action of complexity 
one to that of a Du Val surface. 
\end{abstract}

\maketitle

\section{Introduction}
This article deals with the explicit computation 
of divisor class groups of affine varieties~$X$. 
More precisely, we contribute to the case, when 
$X$ is a quasihomogeneous complete intersection.
Important results in this setting are due to
Flenner~\cite{Fl} showing that isolated 
three-dimensional singularities have free 
abelian divisor class group and Lang, Singh, 
Spiroff, Scheja, Storch~\cite{La,SchSt,SiSP} 
providing explicit descriptions of the divisor 
class groups of affine hypersurfaces 
$X = V(z^n-g)$, where~$g$ is a weighted 
homogeneous polynomial not depending 
on~$z$ and of degree coprime to~$n$. 

We work over 
the field of complex numbers $\CC$.
Our subjects are 
\emph{trinomial varieties}, that means
affine
varieties~$X$ given as the common zero 
locus of trinomials of the form
$$ 
T_0^{l_0} + T_1^{l_1} +T_2^{l_2},
\quad
\theta_1 T_1^{l_1} + T_2^{l_2} +T_3^{l_3},
\
\ldots, \
\theta_{r-2} T_{r-2}^{l_{r-2}} + T_{r-1}^{l_{r-1}} +T_r^{l_r},
$$
with monomials 
$T_i^{l_i} = T_{i1}^{l_{i1}} \cdots T_{in_i}^{l_{in_i}}$ and pairwise different $\theta_i \in \CC^*$.
Any trinomial variety is in particular a
quasihomogeneous complete intersection.
Classical examples are the Pham-Brieskorn
surfaces~\cite{Br}, which also showed up
in~\cite{Mori}.
We refer to~\cite{HaHe,HaHeSue} for the
basic algebraic theory of
trinomial surfaces and to~\cite{Ar2,Ar1}
for recent work on geometric aspects.

Let us turn to the divisor class group
of a trinomial variety~$X$.
For each exponent vector $l_i$ set 
$\mathfrak{l}_i := \gcd(l_{i1}, \ldots, l_{in_i})$. 
Suitably renumbering the variables 
and rearranging the trinomials,
we can always achieve  
$$
\gcd(\mathfrak{l}_0,\mathfrak{l}_1)
\
\ge
\
\gcd(\mathfrak{l}_{i}, \mathfrak{l}_{j})
\text{ for all }
0 \le i < j \le r,
\qquad
\gcd(\mathfrak{l}_0,\mathfrak{l}_1)
\ 
\ge
\ 
\ldots 
\ 
\ge
\
\gcd(\mathfrak{l}_0,\mathfrak{l}_r)
$$
without changing the isomorphy type 
of $X$.
Moreover, we may assume that 
$n_il_{ij} >1$ holds for all
$0 \leq i \leq n$
by eliminating the
variables that occur in a linear term of some relation.
In this situation we call 
$X$ \emph{adjusted}, set
$\mathfrak{l}
:= \gcd(\mathfrak{l}_0, \mathfrak{l}_1,\mathfrak{l}_2)$
and define
$$
c(0)\ 
:=\ 
\gcd(\mathfrak{l}_1,\mathfrak{l}_2), 
\quad
c(1)\ 
:=\ 
\gcd(\mathfrak{l}_0,\mathfrak{l}_2),
\quad
c(2)\ 
:=\ 
\gcd(\mathfrak{l}_0,\mathfrak{l}_1),
$$
$$
c(i)
\ 
:=
\ 
\frac{1}{\mathfrak{l}} 
\gcd(\mathfrak{l}_1,\mathfrak{l}_2)
\gcd(\mathfrak{l}_0,\mathfrak{l}_2)
\gcd(\mathfrak{l}_0,\mathfrak{l}_1)
\quad
\text{for}
\quad
i \geq 3.
$$
Our first main result gives in particular an 
explicit description of the divisor class 
group $\Cl(X)$ of $X$ for the case that 
$\Cl(X)$ is finitely generated.

\goodbreak

\begin{theorem}
\label{thm:divclgr}
Let $X$ be a trinomial variety,
assume that $X$ is adjusted and set
$\tilde{n} := \sum_{i=0}^r ((c(i)-1)n_i-c(i) +1).$
\begin{enumerate}
\item
The divisor class group of $X$ is trivial
if and only if  
$\gcd(\mathfrak{l}_i, \mathfrak{l}_j) = 1$
holds for all $0 \le i < j \le r$.
\item
If
c:=$\gcd(\mathfrak{l}_0, \mathfrak{l}_1) > 1$ and 
$\gcd(\mathfrak{l}_i, \mathfrak{l}_j) = 1$ holds
whenever $j \notin \left\{0, 1\right\}$, then
the divisor class group of $X$ is given by
$$
\quad
\Cl(X) 
\ \cong \ 
\left(\ZZ/ \mathfrak{l}_2\ZZ\right)^{c-1}
\!\times\! \dots \! \times \!
\left(\ZZ/ \mathfrak{l}_r\ZZ\right)^{c-1}
\!\times\! 
\ZZ^{\tilde{n}}.$$ 
\item
If
$\gcd(\mathfrak{l}_0, \mathfrak{l}_1) = \gcd(\mathfrak{l}_0, \mathfrak{l}_2) = \gcd(\mathfrak{l}_1, \mathfrak{l}_2)=2$ and 
$\gcd(\mathfrak{l}_i, \mathfrak{l}_j) = 1$ holds
whenever ${j \notin \left\{0, 1,2\right\}}$,
then the divisor class group of $X$ is given by
$$
\qquad
\Cl(X) 
\ \cong \ 
\ZZ/(\mathfrak{l}_0\mathfrak{l}_1\mathfrak{l}_2/4)\ZZ
\!\times\!
\left(\ZZ/ \mathfrak{l}_3\ZZ\right)^{3}
\!\times\! \dots \! \times \!
\left(\ZZ/ \mathfrak{l}_r\ZZ\right)^{3}
\!\times\!
\ZZ^{\tilde{n}}.$$
\item
If we are not in one of cases~(i) to~(iii), then 
the divisor class group of $X$ is not finitely 
generated.
\end{enumerate}
\end{theorem}

The first assertion of the theorem characterizes factoriality and is~\cite[Thm.~10.1]{HaHe}. Assertions~(ii) and~(iii) are proven in 
Section~\ref{sec:ProofOfTheorem}.
Note that due to~\cite[Cor.~5.8]{ArBrHaWr}, a trinomial 
variety is rational if and only if we are in one of the situations~(i), (ii) or~(iii) of the above theorem; see also Remark \ref{rem:ratcrit}.
Any trinomial variety admits a torus action of 
complexity one and thus is rational if and only 
if its divisor class group is finitely generated,
which in turn yields Assertion~(iv) of the 
theorem.

As immediate consequences of Theorem~\ref{thm:divclgr}, 
we can figure out in Corollaries~\ref{cor:freeab}, \ref{cor:finclgr} 
and~\ref{cor:cyclic} the situations, where the divisor 
class group is free finitely generated, finite
or cyclic.
In particular, we see that among the trinomial
varieties only the affine quadric
$V(T_1^2+T_2^2+T_3^2)$ has divisor class group
of order two, i.e., is half-factorial in
the sense of~\cite{GeHa}.
Moreover, in Corollary~\ref{cor:isosing}, we treat
the case of an isolated singularity.
Finally, observe that any trinomial hypersurface 
as in Situation~(iii) of Theorem \ref{thm:divclgr} goes beyond the framework of 
\cite{La,SchSt,SiSP}; see also Remark~\ref{rem:scheja}.

In the proof of Theorem~\ref{thm:divclgr}, we use
Cox ring theory of rational varieties with torus
action of complexity one.
Recall that the \emph{Cox ring} of a normal variety $X$
with finitely generated divisor class group is
defined as
$$
\mathcal{R}(X)
\ =  \
\bigoplus_{\Cl(X)} \Gamma(X,\mathcal{O}(D)),
$$
where we refer to~\cite{ArDeHaLa} for the
details.
For the Cox ring of a trinomial variety $X$
we provide generators and relations
in~\cite[Prop.~2.6]{HaWr2}.
Based on this, we can compute the divisor
class group of $X$ explicitly.
As a by-product, we obtain a concrete description
of the divisor class group grading
on the Cox ring of a rational trinomial variety;
see Corollary~\ref{cor:grad},
complementing~\cite[Prop.~2.6]{HaWr2}.

Using Corollary~\ref{cor:grad} we observe a new 
feature of the iteration of Cox rings of 
varieties with torus action of complexity one.
We say that a trinomial variety $X$ is 
\emph{hyperplatonic} if 
$\mathfrak{l}_0^{-1} + \ldots + \mathfrak{l}_r^{-1} > r-1$ 
holds.
In this case, reordering 
$\mathfrak{l}_0, \ldots, \mathfrak{l}_r$
decreasingly,
yields $\mathfrak{l}_i = 1$ for all $i \ge 3$
and the triple
$(\mathfrak{l}_0,\mathfrak{l}_1,\mathfrak{l}_2)$
is \emph{platonic}, i.e., one of 
$ 
(5,3,2),
(4,3,2),
(3,3,2),
(x,2,2)
$
and 
$(x,y,1)$,
where 
$x,y \in \ZZ_{\ge 1}$.
We call $(\mathfrak{l}_0,\mathfrak{l}_1,\mathfrak{l}_2)$
the \emph{basic platonic triple} of $X$.
Note that the hyperplatonic trinomial varieties 
comprise all total coordinate spaces
of affine log terminal varieties admitting a torus action of complexity one; 
see~\cite{ArBrHaWr} for the precise statement.
We say that a
variety $X$ admits \emph{iteration of Cox rings}, 
if there exists a chain
\begin{center}
\begin{tikzcd}
X_p
\arrow[r,"\quot H_{p-1}"]
&[0.5cm]
X_{p-1}
\arrow[r,"\quot H_{p-2}"]
&[0.5cm]
\quad
\cdots
\quad
\arrow[r,"\quot H_{2}"]
&[0.5cm]
X_{2}
\arrow[r,"\quot H_{1}"]
&[0.5cm]
X_1:=X,
\end{tikzcd}
\end{center}
where $X_p$ is a factorial affine variety, and in each step, $X_{i+1}$ is the total coordinate
space of $X_i$ and $H_i = \Spec \, \CC[\Cl(X_i)]$ holds.
Due to~\cite[Thm.~1.1]{HaWr2} a rational variety with torus action of complexity one and only constant torus-invariant functions admits iteration of Cox rings if and only if its total coordinate space is factorial or hyperplatonic. In particular any of the occurring total coordinate spaces is again hyperplatonic and
there are exactly the following possible sequences of basic platonic triples arising from Cox ring iterations, see \cite[Cor.~1.4]{HaWr}:
\vspace{2mm}
\begin{enumerate}
\item 
$(1,1,1) \rightarrow (2,2,2) \rightarrow (3,3,2) \rightarrow (4,3,2)$,
\item 
$(1,1,1) \rightarrow (x,x,1) \rightarrow (2x,2,2)$,
\item 
$(1,1,1) \rightarrow (x,x,1) \rightarrow (x,2,2)$,
with $x$ odd,
\item 
$(\mathfrak{l}_{01}^{-1} \mathfrak{l}_0,
\mathfrak{l}_{01}^{-1} \mathfrak{l}_1,1)
\rightarrow
(\mathfrak{l}_0, \mathfrak{l}_1,1)$,
with $\mathfrak{l}_{01} := \gcd(\mathfrak{l}_0, \mathfrak{l}_1) > 1$.
\end{enumerate}
\vspace{2mm}

Note that a hyperplatonic trinomial variety with basic platonic triple $(5,3,2)$ is factorial and therefore does not occur in these chains.
In the above iterations, 
the steps corresponding to  
$(1,1,1)  \rightarrow (x,x,1)$ 
as well as the step of Case~(iv) 
are exactly those steps, where
$H_i$ is a torus.
The remaining parts of the iteration 
chains can be represented by Cox ring 
iterations of Du Val surfaces: 
Any platonic triple $(a,b,c)$ 
defines a Du Val singularity by
$$
Y(a,b,c) 
\ := \ 
\V(T_1^a+T_2^b+T_3^c) 
\ \subseteq \ 
\CC^3.
$$
Case~(i) corresponds to the chain 
$\CC^2 \rightarrow A_1 \rightarrow D_4 \rightarrow E_6$
and $(x,x,1) \rightarrow (x,2,2)$ with $x$ odd resp. 
$(x,x,1) \rightarrow (2x,2,2)$ correspond 
to the chains $\CC^2 \rightarrow A_x$ resp. 
$\CC^2 \rightarrow A_{2x}$.
Overall we obtain the following structural result which we prove in Section \ref{sec:app}.

\begin{theorem}
\label{thm:link2surf}
Let $X$ be a hyperplatonic variety with basic
platonic triple
$(\mathfrak{l}_0, \mathfrak{l}_1, \mathfrak{l}_2)$.
Denote by
$(\mathfrak{l}_0', \mathfrak{l}_1', \mathfrak{l}_2')$
the basic platonic triple of the total coordinate space $X'$ 
of $X$.
Then there is a commutative diagram
$$ 
\xymatrix{
{X'}
\ar[d]_{\git \TT'}
\ar[rr]^{\mathrm{TCS}}
&&
{X}
\ar[d]^{\git \TT}
\\
{Y(\mathfrak{l}_0', \mathfrak{l}_1', \mathfrak{l}_2')}
\ar[rr]_{\mathrm{TCS}}
&&
{Y(\mathfrak{l}_0, \mathfrak{l}_1, \mathfrak{l}_2)},
}
$$
where the horizontal arrows labelled "TCS'' 
are total coordinate spaces 
and the downward arrows are good quotients 
by torus actions. 
\end{theorem}

Finally, as a last application, we
use Theorem~\ref{thm:divclgr} to compute
the divisor class groups of all total
coordinate spaces, i.e.~spectra of the
Cox rings, of normal rational
varieties~$X$ with torus action of complexity one.
Here, Theorem~\ref{thm:divclgr} directly
settles the case of only constant invariant
functions and Corollary~\ref{cor:ClType1} deals
with the remaining case.

\
\\
{\em Acknowledgement:}
Parts of this paper have been written during a stay at Simon Fraser University in Burnaby (BC).
The author is thankful to Nathan Ilten for his kind hospitality.

\tableofcontents

\section{Background on Trinomial varieties}\label{sec:trinomialVarieties}
In this section we recall the basic facts about trinomial varieties as described in the introduction.
We encode a trinomial variety $X$ via a ring $R(A,P_0)$ where $A$ and $P_0$ are matrices storing the coefficients and the exponents of the trinomials defining $X$. The notation is used i.a. in \cite{HaHe,HaWr}, 
where trinomial varieties appear as Cox rings 
of normal rational varieties with torus action of complexity one, i.e. an effective torus action such that the general torus orbit is of codimension one.

\begin{construction}
\label{constr:RAP0}
Fix integers $r,n > 0$, $m \ge 0$ and a partition 
$n = n_0+\ldots+n_r$
with positive integers $n_i$.
For every $i = 0, \ldots, r$, fix a tuple
$l_{i} \in \ZZ_{> 0}^{n_{i}}$ and define a monomial
$$
T_{i}^{l_{i}}
\ := \
T_{i1}^{l_{i1}} \cdots T_{in_{i}}^{l_{in_{i}}}
\ \in \
\CC[T_{ij},S_{k}; \ 0 \le i \le r, \ 1 \le j \le n_{i}, \ 1 \le k \le m].
$$
We will also write $\CC[T_{ij},S_{k}]$ for the 
above polynomial ring. 
Let $A:= (a_0, \ldots, a_r)$ be a $2 \times (r+1)$~matrix with pairwise 
linearly independent columns $a_i \in \CC^2$. 
For every $i = 0, \dots, r-2$ we define
$$
g_{i}
\ :=  \
\det
\left[
\begin{array}{lll}
T_i^{l_i} & T_{i+1}^{l_{i+1}} & T_{i+2}^{l_{i+2}}
\\
a_i & a_{i+1}& a_{i+2}
\end{array}
\right]
\ \in \
\CC[T_{ij},S_{k}].
$$
We build up an $r \times (n+m)$~matrix 
from the exponent vectors $l_0, \ldots, l_r$ of these 
polynomials:
$$
P_{0}
\ := \
\left[
\begin{array}{ccccccc}
-l_{0} & l_{1} &  & 0 & 0  &  \ldots & 0
\\
\vdots & \vdots & \ddots & \vdots & \vdots &  & \vdots
\\
-l_{0} & 0 &  & l_{r} & 0  &  \ldots & 0
\end{array}
\right].
$$
Denote by $P_0^*$ the transpose of $P_0$ and consider 
the projection
$$
Q \colon \ZZ^{n+m} 
\ \to \ 
K_{0} 
\ := \ 
\ZZ^{n+m}/\mathrm{im}(P_{0}^{*}).
$$
Denote by $e_{ij},e_{k} \in \ZZ^{n+m}$ the canonical
basis vectors corresponding to the variables 
$T_{ij}$, $S_{k}$.
Define a $K_0$-grading on $\CC[T_{ij},S_{k}]$ 
by setting
$$
\deg(T_{ij}) \ := \ Q(e_{ij}) \ \in \ K_{0},
\qquad
\deg(S_{k}) \ := \ Q(e_{k}) \ \in \ K_{0}.
$$
This is the finest possible grading of
$\CC[T_{ij},S_{k}]$ leaving the variables 
and the $g_i$ homogeneous
and any other such grading coarsens this maximal one.
In particular, we have  a $K_{0}$-graded 
factor algebra
$$
R(A,P_{0})
\ := \
\CC[T_{ij},S_{k}] / \bangle{g_{0}, \dots, g_{r-2}}.
$$
\end{construction}

Note that every trinomial variety can be realized as the spectrum of a ring $R(A,P_0)$. Due to \cite{HaHeSue} these rings are normal complete intersections and admit only constant homogeneous units.
The variables $T_{ij}$ and $S_k$ are \emph{$K_0$-prime}, i.e. a homogeneous non-zero non-unit which, whenever it divides a product of homogeneous elements, it also divides one of the factors.
In general the rings $R(A,P_0)$ are not unique factorization domains but have the following similar but weaker property:
They are \emph{$K_0$-factorial}, i.e. every non-zero homogeneous non-unit is a product of $K_0$-primes.

We will make frequent use of the following rationality criterion for trinomial varieties
from \cite[Cor.~5.8]{ArBrHaWr}:

\begin{remark}\label{rem:ratcrit}
Let $X:=\Spec \, R(A,P_0)$ be an adjusted trinomial variety and set ${\mathfrak{l}_i := \gcd(l_{i1}, \ldots, l_{in_i})}$.
Then $X$ is rational if 
and only if one of the following conditions
holds:
\begin{enumerate}
\item
We have $\gcd(\mathfrak{l}_i,\mathfrak{l}_j) = 1$
for all $0 \le i < j \le r$, in other words, $R(A,P_0)$ 
is factorial.
\item
We have
$\gcd(\mathfrak{l}_0,\mathfrak{l}_1) > 1$
and $\gcd(\mathfrak{l}_i,\mathfrak{l}_j) = 1$ 
holds
whenever $j \notin \left\{0,1\right\}$.
\item
We have
$
\gcd(\mathfrak{l}_0,\mathfrak{l}_1) = 
\gcd(\mathfrak{l}_0,\mathfrak{l}_2) =
\gcd(\mathfrak{l}_1,\mathfrak{l}_2) =
2
$ 
and $\gcd(\mathfrak{l}_i,\mathfrak{l}_j) = 1$ 
holds
whenever $j \not\in \{0,1,2\}$.
\end{enumerate} 
\end{remark}

In order to prove our main results we make use of
the explicit description of the total coordinate space, i.e. the spectrum of the Cox ring, of a
rational trinomial variety given in \cite[Prop. 2.6]{HaWr2}. We recall basic facts about Cox rings and their geometric meaning in general and then specialize to Cox rings of trinomial varieties.

Let $X$ be a normal variety with only constant invertible global functions and finitely generated divisor class group $\Cl(X)$.
The Cox ring of $X$ is the $\Cl(X)$-graded ring 
$$ 
\mathcal{R}(X)
\ =  \
\bigoplus_{\Cl(X)} \Gamma(X,\mathcal{O}(D)),
$$
where $\Gamma(X,\mathcal{O}(D))$ is the sheaf of global sections of the divisorial algebra defined by the divisor class of $D$; see \cite{ArDeHaLa} for the details. 
The grading of the divisor class group defines a quasitorus action of the {\em characteristic quasitorus}  $H:=\Spec\ \CC[\Cl(X)]$ on the {\em total coordinate space} $\bar{X}:=\Spec\ \RRR(X)$.
Moreover $X$ can be regained as a good quotient $\hat X \rightarrow X$ by the quasitorus action of $H$ on an invariant big open subset $\hat X\subseteq \bar X$. If $X$ is affine, then $\hat X = \bar X$ holds and one has $\RRR(X)^H\cong \OOO(X)$.

Computing the Cox ring of a variety in terms of generators and relations is often hardly possible. However, for rational trinomial varieties there is an explicit description: 
Let $X :=\Spec \, R(A,P_0)$ be an adjusted rational trinomial variety and set $\mathfrak{l} := \gcd(\mathfrak{l}_0,\mathfrak{l}_1,\mathfrak{l}_2)$.
Then, due to \cite[Lemma 2.5]{HaWr2},
the number~$c(i)$ of irreducible components 
of $V(X,T_{ij})$, where $j = 1,\ldots, n_i$, 
is given by
\begin{center}
\renewcommand{\arraystretch}{1.8} 
\begin{tabular}{c|c|c|c|c}
$i$ & $0$ & $1$ & $2$ & $\ge 3$
\\
\hline
$c(i)$
& 
$\gcd(\mathfrak{l}_1,\mathfrak{l}_2)$
& 
$\gcd(\mathfrak{l}_0,\mathfrak{l}_2)$
& 
$\gcd(\mathfrak{l}_0,\mathfrak{l}_1)$
&
$\frac{1}{\mathfrak{l}} 
\gcd(\mathfrak{l}_1,\mathfrak{l}_2)
\gcd(\mathfrak{l}_0,\mathfrak{l}_2)
\gcd(\mathfrak{l}_0,\mathfrak{l}_1)$
\end{tabular}
\end{center}
and the total coordinate space can be computed in the following way:

\begin{proposition}\label{prop:coxRAP0}\cite[Prop. 2.6]{HaWr2}
\label{prop::isotropy}
Let $R(A,P_0)$ be a non-factorial
ring defining an adjusted rational trinomial variety $X:=\Spec \, R(A,P_0)$.
Set
$$
P_1
\ := \
\left[
\begin{array}{ccccccccc}
\frac{-1}{\gcd(\mathfrak{l}_0,\mathfrak{l}_1)} l_0
& 
\frac{1}{\gcd(\mathfrak{l}_0,\mathfrak{l}_1)} l_1
& 0 &  \dots & &0 & 0 & \dots & 0
\\[5pt]
\frac{-1}{\gcd(\mathfrak{l}_0,\mathfrak{l}_2)} l_0
& 0 
& 
\frac{1}{\gcd(\mathfrak{l}_0,\mathfrak{l}_2)} l_2 
& 0 & & 0& & &\\
-l_0 & 0 &  & l_3 & & 0 &\vdots &&\vdots
\\
\vdots &  &  & \vdots & \ddots & \vdots && &
\\
-l_0 & 0 & \dots & 0      &        & l_r & 0& \dots & 0
\end{array}
\right],
$$
let $c(i)$ be as above and define numbers 
\begin{center}
$n' \ 
:=\ 
c(0)n_0 + \ldots + c(r)n_r,\qquad
r' 
\ 
:=
\ 
c(0)+ \ldots + c(r)-1$,
\\
$n_{i,1}, \ldots, n_{i,c(i)} \ 
:= 
\ n_i,\qquad 
l_{ij,1}, \ldots, l_{ij,c(i)}
\ 
:= 
\ 
\gcd ( (P_{1})_{1,ij},\ldots,(P_{1})_{r,ij} ).
$
\end{center}
Build up an
$r'\times (n'+m)$-matrix
$P_0'$ as in Construction \ref{constr:RAP0}
using the vectors
$l_{i,t} := 
(l_{i1,t}, \ldots, l_{in_i,t}) \in \ZZ^{n_{i,t}}$.
Then, choosing a suitable matrix $A'$,
the affine variety $\Spec \, R(A',P'_0)$ is the 
total coordinate space of $X$.
\end{proposition}

In particular the total coordinate space of a trinomial variety is again a trinomial variety and the ring $R(A',P_0')$ as constructed above is a factor ring of the polynomial ring
$$\CC[T_{ij,t}, S_k'; \ 1 \leq i \leq r,\ 1 \leq j \leq n_i,\ 1 \leq t \leq c(i), \ 1 \leq k \leq m].$$
Moreover
the grading of
$R(A', P_0')$ with respect to the divisor class group of $X$ is given as follows: 
For fixed $0 \leq i \leq r$ and $1 \leq j \leq n_i$ 
let $D_{ij,1}, \hdots D_{ij,c(i)}$ be the prime divisors inside 
$\V(X;T_{ij})$.
Then due to \cite[Thm. 1.3]{HaSu} and \cite[Proof of Prop. 6.6]{ArBrHaWr} the variables $T_{ij,t}$ in $R(A',P_0')$ are of degree  
$$\deg(T_{ij,t})\ =\ [D_{ij,t}]\ 
\in\ 
\Cl(X).$$
Now consider the prime divisor $\V(X ;S_k) = E_k$.
Then, as $S_k$ is a $K_0$-prime element in the $K_0$-factorial ring $R(A,P_0)$, the prime divisor $E_k$ is principal due to \cite[Prop. 1.5.3.3]{ArDeHaLa}. Thus using \cite[Thm. 1.3]{HaSu} and \cite[Proof of Prop. 6.6]{ArBrHaWr} we obtain
$$
\deg(S_k') \ =\  [E_k]\ =\ [0] \ \in\  \Cl(X).
$$
Note that the degrees of the variables $T_{ij,t}$ and $S_k'$ 
generate the divisor class group of $X$.

Due to \cite[Thm. 1.7]{HaWr} 
the grading on $R(A', P_0')$ defined by the divisor class group
is a downgrading of the grading defined via $P_0'$ 
as in Construction \ref{constr:RAP0}. It can be obtained in the following way:

\begin{construction}
\label{constr:RAPdown}
Let $R(A,P_0)$ be a ring as in Construction \ref{constr:RAP0}.
Choose an integral
$s \times (n + m)$ matrix $d$ and build the 
$(r+s) \times (n + m)$ stack matrix
$$
P 
\ := \
\left[
\begin{array}{c}
P_0
\\
d
\end{array}
\right].
$$
We require the columns of $P$ to be pairwise 
different primitive vectors generating
$\QQ^{r+s}$ as a vector space. 
Let $P^*$ denote the transpose of $P$ and 
consider the projection
$$
Q \colon 
\ZZ^{n+m} 
\ \to \ 
K 
\ := \ 
\ZZ^{n+m} / \mathrm{im}(P^*).
$$
Denoting as before by $e_{ij}, e_k \in \ZZ^{n+m}$ the 
canonical basis vectors corresponding to
the variables $T_{ij}$ and $S_k$, we obtain a 
$K$-grading on $\CC[T_{ij}, S_k]$ by setting
$$
\deg(T_{ij}) \ := \ Q(e_{ij} ) \ \in \ K,
\qquad\qquad
\deg(S_k) \ := \ Q(e_k) \ \in \ K.
$$
This $K$-grading coarsens the $K_0$-grading of 
$\CC[T_{ij},S_k ]$ given in Construction~\ref{constr:RAP0}
and thus defines a grading on $R(A,P_0)$. We denote the ring $R(A,P_0)$ endowed with this grading by $R(A,P)$.
\end{construction}

As a first glimpse towards the computation of the divisor class group of a rational trinomial variety $X$ we describe a specific
finite subgroup of $\Cl(X)$ 
which necessarily appears whenever the Cox ring of $X$ is realized as in Construction \ref{constr:RAPdown}:

\begin{construction}
Let $X$ be any variety having
a graded ring $R(A,P)$ as in Construction \ref{constr:RAPdown} as its Cox ring.
Set $K_0:= \ZZ^{n+m}/\mathrm{im}(P_0)^*$ and denote by
$(K_0)^{\mathrm{tors}}$ 
the torsion subgroup of $K_0$.
Then the group 
$$(K_0)^{\mathrm{tors}}\  \subseteq\ \ZZ^{n+m}/\im(P^*)\ =\ K\ \cong\ \Cl(X)$$
is a finite subgroup of $\Cl(X)$ due to injectivity of $P^*$. 
We call $\Cl(X)^{\mathrm{ctors}}:=(K_0)^{\mathrm{tors}}$ the \emph{compulsory torsion of $\Cl(X)$}.  
\end{construction}

\section{Proof of Theorem \ref{thm:divclgr}}
\label{sec:ProofOfTheorem}
This section is dedicated to the proof of Theorem \ref{thm:divclgr}.
The idea is to approximate the divisor class group of a trinomial variety $X$ in the following sense: In a first step we compute in Proposition \ref{prop:rkClGroup} and Lemma \ref{lem:comptors} the rank and the compulsory torsion of $\Cl(X)$ to determine a subgroup of $\Cl(X)$.
In a second step we fix a subgroup of the group of Weil divisors of $X$ mapping surjectively onto $\Cl(X)$ and investigate the kernel of this map; 
this includes i.a. computing the common $\Cl(X)$-degree of the defining relations of the Cox ring of $X$, see Proposition \ref{prop:zerodegree}. 
This enables us to realize $\Cl(X)$ as a factor group of a group of the same rank and finally determine $\Cl(X)$ as stated in Theorem \ref{thm:divclgr}.
We work in the notation of Section \ref{sec:trinomialVarieties}.

\begin{proposition}
\label{prop:rkClGroup}
Let $X:= \Spec \, R(A,P_0)$ be an adjusted non-factorial rational trinomial variety. Then the divisor class group $\Cl(X)$ is of rank
$$ 
\sum_{i=0}^r ((c(i)-1)n_i-c(i) +1).
$$
\end{proposition}
\begin{proof}
The rank of the divisor class group of $X$ equals the difference between the dimension of $X$ and the dimension of its total coordinate space
$\overline{X}:= \Spec \, R(A', P_0')$ as described in
Proposition \ref{prop:coxRAP0}. We conclude
\begin{align*}
\mathrm{rk}(\Cl(X)) \ =\  \dim(\overline{X}) - \dim(X) 
&\ =\  n'-(r'-1) - (n-(r-1)) \\
&\ =\  \sum_{i=0}^r ((c(i)-1)n_i-c(i) +1).
\end{align*}
\end{proof}

\begin{lemma}\label{lem:comptors} 
Let $X:=\Spec \, R(A,P_0)$ be an adjusted non-factorial rational
trinomial variety.
\begin{enumerate}
\item
If $c:=\gcd(\mathfrak{l}_0, \mathfrak{l}_1) >1$ and 
$\gcd(\mathfrak{l}_i, \mathfrak{l}_j) = 1$ holds
whenever $j \notin \left\{0, 1\right\}$, then $X$ has compulsory torsion isomorphic to 
$$
\left(\ZZ/ \mathfrak{l}_2\ZZ\right)^{c-1}
\times \dots  \times 
\left(\ZZ/ \mathfrak{l}_r\ZZ\right)^{c-1}.$$

\item
If $\gcd(\mathfrak{l}_0, \mathfrak{l}_1) = \gcd(\mathfrak{l}_1, \mathfrak{l}_2) = \gcd(\mathfrak{l}_0, \mathfrak{l}_2)=2$ and 
$\gcd(\mathfrak{l}_i, \mathfrak{l}_j) = 1$ holds
whenever ${j \notin \left\{0, 1,2\right\}}$, then $X$ has compulsory torsion isomorphic to 
$$
\ZZ/(\mathfrak{l}_0/2) \ZZ
\times
\ZZ/ (\mathfrak{l}_1/2)\ZZ
\times
\ZZ/(\mathfrak{l}_2/2)\ZZ
\times
\left(\ZZ/\mathfrak{l}_3\ZZ\right)^{3}
\times \dots  \times 
\left(\ZZ/\mathfrak{l}_r\ZZ\right)^{3}.$$
\end{enumerate}

\begin{proof}
We prove (i).
With our subsequent considerations we obtain that the divisor class group of $X$
is given as $\ZZ^{n'+m}/\im ((P')^*)$, where $P'$ is some $(r'+s') \times (n'+m)$ stack matrix
$$
\left[
\begin{array}{c}
P_0'
\\
d'
\end{array}
\right],
$$
of full row rank. Moreover due to Proposition \ref{prop::isotropy} the matrix
$P_0'$ is the $r' \times (n'+m)$ matrix build up
by the exponent vectors $c^{-1}l_0$,
$c^{-1}l_1$ and $c$ copies
$l_{i,1}, \dots, l_{i,c}$ of $l_i$ for $i \geq 2$. 
Thus, to obtain the assertion,
we compute the elementary divisors of $P_0'$:
Suitable elementary column operations transform $P_0'$
into
$$
\left[
\begin{array}{cccccccc}
c^{-1}\mathfrak{l}_0
& 
c^{-1}\mathfrak{l}_1
& 0 &  \dots &0 & 0 & \dots & 0
\\[5pt]
c^{-1}\mathfrak{l}_0
& 0 
& 
\mathfrak{l}_{2,1}
& & 0& & &\\
\vdots &  &  & \ddots & \vdots && &
\\
c^{-1}\mathfrak{l}_0 
& 0 & \dots  &        & \mathfrak{l}_{r,c} & 0& \dots & 0
\end{array}
\right].
$$
As $\gcd(\mathfrak{l}_i, \mathfrak{l}_j) = 1$
holds for $i,j \notin \left\{0,1\right\}$
we obtain for $1 \leq t \leq c$ that the 
$(r'-t+1)$-th determinantal divisor of $P_0'$ equals
$\mathfrak{l}_2^{c-t} \cdots \mathfrak{l}_r^{c-t}$.
This proves the assertion.

For the proof of (ii)
we note that in this case $P_0'$
is built up
by $2$ copies of 
$1/2l_{0}, 1/2 l_1$ and $1/2 l_2$
and 
$4$ copies of each term
$l_i$ for $i \geq 3$. 
Then, applying the same 
arguments as above,
we obtain the assertion.
\end{proof}
\end{lemma}

\begin{proposition}
\label{prop:zerodegree}
Let ${X:= \Spec \, R(A,P_0)}$ be an adjusted non-factorial rational trinomial variety with total coordinate space $\Spec \, R(A',P_0')$. 
\begin{enumerate}
    \item 
    If $\gcd(\mathfrak{l}_0, \mathfrak{l}_1) >1$ and
    $\gcd(\mathfrak{l}_i, \mathfrak{l}_j) = 1$ holds, whenever $j \notin \left\{0, 1\right\}$, then the defining relations of $R(A',P_0')$ 
    have $\Cl(X)$-degree zero.
    \item
    If $\gcd(\mathfrak{l}_0, \mathfrak{l}_1) = \gcd(\mathfrak{l}_1, \mathfrak{l}_2) = \gcd(\mathfrak{l}_0, \mathfrak{l}_2)=2$ and 
    $\gcd(\mathfrak{l}_i, \mathfrak{l}_j) = 1$ holds
    whenever ${j \notin \left\{0, 1,2\right\}}$,  then the degree of the defining relations of $R(A',P_0')$ generates a subgroup of order two in $\Cl(X)$.
\end{enumerate}
\end{proposition}

\begin{definition}
Let $X$ be an irreducible normal variety and $Y \subseteq X$ a prime divisor. 
Let furthermore $\mathfrak{A}:= \bangle{f_1, \dots, f_r} \leq \mathcal{O}(X)$ be any ideal. 
Then we define the \emph{order of $\mathfrak{A}$ along $Y$} to be the minimum
$\min(\mathrm{ord}_Y(f_i); \ i = 1, \dots, r)=: \mathrm{ord}_Y( \mathfrak{A})$.
\end{definition}

\begin{lemma}
\label{lem:princIdeal}
Let $X$ be an irreducible normal variety, 
$\mathfrak{A}:= \bangle{f_1, \dots, f_r} \leq \mathcal{O}(X)$ any ideal 
and $f \in \mathcal{O}(X)$.
Then the following statements are equivalent:
\begin{enumerate}
\item $\mathrm{ord}_Y(\mathfrak{A}) = \mathrm{ord}_Y(f)$ holds for all prime divisors $Y \subseteq X$.
\item  $\bangle{f} = \mathfrak{A}$ holds, i.e. $\mathfrak{A}$ is a principal ideal.
\end{enumerate}
In particular the Weil divisor $D:= \sum \mathrm{ord}_Y( \mathfrak{A})$, where the sum runs over all 
prime divisors $Y \subseteq X$, is principal if and only if $\mathfrak{A}$ is a principal ideal.
\begin{proof}
We prove (i) $\Rightarrow$ (ii).
Observe that $f\mid f_i$ holds for $i = 1, \dots, r$ as $\div(f) \leq \div(f_i)$ by construction.
In particular $\bangle{f} \supseteq \mathfrak{A}$. We prove the other inclusion.
Consider the covering $X=U_1\cup\ldots\cup U_r$, where 
$$U_i\ :=\ X \setminus (Y_{i_1} \cup \dots \cup Y_{i_{k_i}}),$$
where all prime divisors $Y$ with $\mathrm{ord}_{Y}(f_i) \neq \mathrm{ord}_Y(\mathfrak{A})$ 
occur among the $Y_{i_t}$. Then inside $U_i$ we have $f_i \mid f$. 
Thus $c_i \cdot f_i = f$ holds with $c_i \in \mathcal{O}(U)^*$. Considering the associated sheaf 
$\tilde{\mathfrak{A}}$ of $\mathfrak{A}$ we obtain $f \in \tilde{\mathfrak{A}}(X) = \mathfrak{A}$. 
The other implication is clear.
\end{proof}
\end{lemma}

\begin{lemma}
\label{lem:principal}
Let ${X:= \Spec \, R(A,P_0)}$ be an adjusted non-factorial rational trinomial variety
and assume
$\gcd(\mathfrak{l}_0, \mathfrak{l}_1) = \gcd(\mathfrak{l}_1, \mathfrak{l}_2) = \gcd(\mathfrak{l}_0, \mathfrak{l}_2)=2$ and 
$\gcd(\mathfrak{l}_i, \mathfrak{l}_j) = 1$ holds,
whenever $j \notin \left\{0,1,2\right\}$.
Let $D_{0j,1}$ and $D_{0j,2}$ denote the irreducible components of $\V(X;T_{0j})$.
Then, for every
$y \in \ZZ_{\ge 0}$ with $y \mid \mathfrak{l}_0$, we 
obtain a cyclic subgroup of order $y$ in the divisor 
class group:
$$ 
\langle D_y \rangle 
\ \subseteq \
\Cl(X),
\qquad
\text{ with }\quad
D_y
\ := \ 
\sum_{j=1}^{n_0}\frac{1}{y}l_{0j} D_{0j,1}.
$$

\begin{proof}
Consider the ideal 
$\mathfrak{A}_y \subseteq R(A,P)$ generated by $T_1^{1/2l_1} + i\cdot T_{2}^{1/2l_2}$
and $T_0^{1/y\cdot l_0}$.
Then, as the irreducible components $D_{0j,1}$ and $D_{0j,2}$ of $\V(X; T_{0j})$ are of the form
$$D_{0j,1} 
\ 
=\ 
\V(T_{0j}, T_1^{1/2l_1}+ i \cdot T_2^{1/2l_2})
\quad \text{and} \quad 
D_{0j,2}
\ 
=\ 
\V(T_{0j}, T_1^{1/2l_1}- i \cdot T_2^{1/2l_2}),$$
we have
$D_y = \sum_{Y} \mathrm{ord}_Y( \mathfrak{A}_y)$,
where $Y$ runs over all prime divisors of $X$.
Applying Lemma \ref{lem:princIdeal} it thus suffices to show that the ideal $\mathfrak{A}_y$ is principal if and only if $y =1$ holds.
Note that $\mathfrak{A}_1 = \bangle{T_1^{1/2l_1}+ iT_2^{1/2l_2}}$ holds in $R(A,P_0)$.
So let $y \neq 1$ and assume there is an $f \in \mathfrak{A}_y$ with $\bangle{f} = \mathfrak{A}$. 
Then there exist $g_1, g_2, h_1, h_2 \in \CC[T_{ij}, S_k]$ with
$g_1 \cdot f + I = T_0^{1/y \cdot l_0} + I$ 
and 
$g_2 \cdot f +I = T_1^{1/2l_1}+i T_2^{1/2l_2}+I$
and
$$h_1\cdot  T_0^{1/y \cdot l_0}+ h_2 \cdot (T_1^{1/2l_1}+i T_2^{1/2l_2})+ I\ =\ f+I.$$
Inserting the third formula into the first one we obtain 
$$T_0^{1/y\cdot l_0} + I\ =\ g_1\cdot h_1 \cdot  T_0^{1/y \cdot l_0}+ g_1 \cdot h_2 \cdot (T_1^{1/2l_1}+i T_2^{1/2l_2})+ I$$
and so in particular
\begin{equation}\label{equ5}
h
\ 
:=
\ 
(g_1\cdot h_1-1) \cdot  T_0^{1/y \cdot l_0}+ g_1 \cdot h_2 \cdot (T_1^{1/2l_1}+i T_2^{1/2l_2})\ \in\  I.
\end{equation}
As there can not occur any term $T_0^{1/y \cdot l_0}$ in $I$ for $y \neq 1$, 
we conclude that $g_1$ and $h_1$ each have a constant term. 
Inserting the third formula above into the second,
we obtain a constant term in $g_2$ and $h_2$ with similar arguments. 
But this leads to a term ${\lambda \cdot(T_1^{1/2l_1} + i \cdot T_2^{1/2l_2})}$
with $\lambda \neq 0$ in (\ref{equ5}); a contradiction to $h \in I$.
\end{proof}
\end{lemma}

\begin{proof}[Proof of Proposition \ref{prop:zerodegree}]
We prove (i). As $c(0) = \gcd(\mathfrak{l}_1, \mathfrak{l_2}) = 1$ holds the vanishing set
$\V(X; T_{0j})$ is a prime divisor for $j=1, \dots, n_0$. 
In particular, as $R(A,P_0)$ is $K_0$-factorial, $K_0$-primality of the variable $T_{0j}$ implies
that $D_{0j,1}$ is a principal divisor for $j = 1, \ldots, n_0$
due to \cite[Prop. 1.5.3.3]{ArDeHaLa}. 
We conclude
$$\deg(T_{0,1}^{l_{0,1}})\ =\ \sum_{j = 1}^{n_0} l_{0j,1} [D_{0j,1}]\ =\  [0] \ \in\  \Cl(X).$$
This proves the assertion as all defining relations have the same degree. 

We turn to (ii).
Due to Proposition \ref{prop:coxRAP0} we have
$$\deg(T_{0,1}^{l_{0,1}})\ =\ \sum_{j = 0}^{n_0} l_{0j,1}[D_{0j,1}]\ =\  \sum_{j = 0}^{n_0} \frac{1}{2}l_{0j}[D_{0j,1}] \ \in\  \Cl(X).$$
Thus, applying Lemma \ref{lem:principal} and using the fact
that all defining relations have the same degree, 
the assertion follows.
\end{proof}

The following Construction shows how to use the  above Proposition in order to realize $\Cl(X)$  as a factor group of a finitely generated group:

\begin{construction}\label{constr:ClRel}
Let $X := \Spec \,  R(A,P_0)$ be a rational trinomial variety with total coordinate space $\Spec \, R(A', P_0')$. Then the divisor class group of $X$ is generated by the $\Cl(X)$-degrees of the variables $T_{ij,t}$ and $S_k'$.
Consider the map $\pi$ sending each Weil divisor $D$ to its class $[D] \in \Cl(X)$.
Then, as $\deg(S_k') = [0] \in \Cl(X)$ holds, the subgroup
$$
\mathfrak{D}\ :=\ \bangle{D_{ij,t}; \ 0 \leq i \leq r, \ 1 \leq j \leq n_i, 1\leq t \leq c(i)} \ \subseteq\  \mathrm{WDiv}(X),
$$
which is isomorphic to $\ZZ^{n'}$, is mapped surjectively onto $\Cl(X)$ via $\pi$.
Now let $(D_\alpha)_\alpha$ be any finite collection of principal Weil divisors inside $\mathfrak{D}$, 
i.e. we have
$$[D_\alpha]\ =\ \sum a^\alpha_{ij,t} [D_{ij,t}]\ =\ [0] \ \in\  \Cl(X).$$
Then writing the coefficients $a^\alpha_{ij,t}$ as rows into an $\alpha \times n'$ matrix $P$
we obtain a commutative diagram
$$ 
\xymatrix{
\ZZ^{n'}\ar[rr]^{e_{ij,t} \mapsto [D_{ij,t}]}
\ar[rd]
&
&
\Cl(X)
\\
&
\ZZ^{n'}/\im(P^*). \ar[ur]
&
}
$$
In particular $\Cl(X)$ is isomorphic to a factor group of $\ZZ^{n'}/ \im(P^*)$.
\end{construction}

\begin{lemma}\label{lemma:matrix}
Let $l_i \in \ZZ_{>0}^{n_i}$ be any tuple, $k \in \ZZ_{\geq1}$ an integer, denote by $E_{n_i}$ the unit matrix of size $n_i$ and set
$$
A(k, l_i)\ :=\ 
\left[
\begin{array}{ccc}
l_i
& 
\dots
& 
0
\\
\vdots
& 
\ddots 
& 
\vdots
\\
0 
& 
\dots
&
l_i
\\
E_{n_i} 
&
\dots
&
E_{n_i}
\end{array}
\right]
\ \in \ \Mat(k+n_i,k\cdot n_i, \ZZ).
$$
Then $A(k, l_i)$ has rank $n_i -1 + k$ and the $(n_i-1+k)$-th determinantal
divisor divides $\mathfrak{l}_i^{k-1}$, where $\mathfrak{l}_i := \gcd(l_{i1}, \dots,  l_{in_i})$.
\begin{proof}
Choose for any $2 \leq t \leq k$ an integer $1 \leq j_t \leq n_i$ and
denote by $e_{j_t}$ the column vector having $1$ as $j_t$-th entry and 
all other entries equal zero. 
Consider the following $(n_i-1+k)\times (n_i-1+k)$ square matrix 
obtained by deleting the first row and several of the last $(k-1)\cdot n_i$ columns 
of $A(k,l_i)$
$$
\left[
\begin{array}{cccccc}
0 & \dots & 0
&
l_{ij_2}
& 
\dots
& 
0
\\
\vdots & & \vdots
&
\vdots
& 
\ddots 
& 
\vdots
\\
0 & \hdots & 0
&
0 
& 
\dots
&
l_{ij_k}
\\
 & E_{n_i} & 
&
e_{j_2}
&
\dots
&
e_{j_k}
\end{array}
\right].
$$
The determinant of this matrix equals the product
$l_{ij_2} \cdots l_{ij_k}$ up to sign. 
With ${\mathfrak{l}_i = \gcd (l_{i1}, \dots, l_{in_i})}$ we obtain
$$\gcd( \prod_{t=2}^k l_{ij_t}; \  j_t \in \left\{1, \dots, n_i\right\})\ =\  \mathfrak{l}_i^{k-1}.$$
This shows that the $(n_i-1+k)$-th determinantal divisor divides $\mathfrak{l}_i^{k-1}$.
Moreover, as $A(k,l_i)$ is obviously not of full rank this proves the assertions.
\end{proof}
\end{lemma}

\begin{proof}[Proof of Theorem \ref{thm:divclgr}]
Using the rationality and the factoriality criterion as stated in Remark \ref{rem:ratcrit} only in the cases (ii) and (iii) there is something to show.

We prove Case (ii).
Let $D_{ij,1} \cup \dots \cup D_{ij,c(i)}$ be the decomposition of $\V(X,T_{ij})$ into prime divisors. 
As $R(A,P_0)$ is $K_0$-factorial and $T_{ij}$ is $K_0$-prime, we can apply
\cite[Prop. 1.5.3.3]{ArDeHaLa} and obtain
\begin{equation}\label{equ1}
\sum_{t=1}^{c(i)} [D_{ij,t}] \ =\  [0] \ \in\  \Cl(X).
\end{equation}
Moreover, due to Proposition \ref{prop:zerodegree} (i) the defining relations of $R(A',P_0')$ have degree zero. 
In particular, due to Proposition \ref{prop::isotropy} for every $i=0, \ldots, r$ and $1 \leq t \leq c(i)$ we obtain a term 
$T_{i,t}^{l_{i,t}}= T_{i1,t}^{l_{ij,t}} \cdots T_{ij,t}^{l_{in_i,t}}$ of degree zero occurring in the relations of $R(A',P_0')$.
This gives rise to relations
\begin{equation}\label{equ2}
\sum_{j=1}^{n_i} l_{ij,t} [D_{ij,t}] \ =\  [0] \ \in\  \Cl(X),
\end{equation}
where $i = 0, \dots, r$ and $t=1, \dots, c(i)$.
As $l_{i, 1}= \dots = l_{i,c(i)}$ holds for any $i=0, \dots, r$, 
the relations (\ref{equ1}) and (\ref{equ2}) give rise to 
block matrices $A(c(i), l_{i,1})$ in a matrix $P$ as in Construction \ref{constr:ClRel}.
In particular we get an $m' \times n'$ matrix with 
$m':=\sum_{i=0}^r (n_i+c(i))$ and $n':= \sum_{i=0}^r c(i)\cdot n_i$ 
of the following form
\begin{equation}\label{equ:P}
P\ :=\ 
\left[
\begin{array}{cccc}
 A(c(0), l_{0,1})&0& \cdots &0
 \\
 0 &  A(c(1), l_{1,1}) & \cdots & 0
\\
 \vdots & \vdots & \ddots & \vdots
 \\
 0 & 0 & \hdots &  A(c(r), l_{r,1})
\end{array}
\right].
\end{equation}
Note that $P$ is of rank $\sum_{i=0}^r (n_i -1+c(i))$
and the 
$\mathrm{rk}(P)$-th determinantal divisor of $P$ equals the product of the 
$(n_i-1+c(i))$-th determinantal divisors of the block matrices $A(c(i),l_{i,1})$. 
Moreover, $c(0) = c(1) = 1$ and $c:=c(i) = \gcd(\mathfrak{l}_0, \mathfrak{l}_1)$ holds and due Proposition \ref{prop:coxRAP0} we have $l_{i,1} = \ldots =  l_{i,c} =  l_i$ for $i \geq 2$.
Therefore, applying Lemma \ref{lemma:matrix} we conclude 
that the divisor class group of $X$ is isomorphic to a factor group 
of the group
\begin{equation}\label{equ4}
\ZZ^{n'}/\im(P^*) \ \cong\  \ZZ^{n'-\mathrm{rk}(P)} \times G,
\end{equation}
where $G$ is a finite abelian group of order $k$ with
$k\mid(\mathfrak{l}_{2}^{c-1} \cdots \mathfrak{l}_{r}^{c-1})$.
As $\mathrm{rk}(\Cl(X)) = n' - \mathrm{rk}(P)$
holds due to Proposition \ref{prop:rkClGroup} and $\Cl(X)^{\mathrm{ctors}}$
is a subgroup of $\Cl(X)$
of order $(\mathfrak{l}_{2}^{c-1} \cdots \mathfrak{l}_{r}^{c-1})$
we obtain 
$\Cl(X) \cong \ZZ^{\tilde{n}} \times \Cl(X)^{\mathrm{ctors}}$.

We prove Case (iii).
With the same arguments as in Case (ii) we get relations of the form
(\ref{equ1}). Moreover, since the degrees of the relations and thus the degree of all terms occurring
in the Cox ring $R(A',P_0')$ of 
$X$ coincide, we obtain
\begin{equation}\label{equ3}
\sum_{j=1}^{n_0} l_{0j,1} [D_{0j,1}]  \ =\ \sum_{j=1}^{n_i} l_{ij,t} [D_{ij,t}] \ \in\  \Cl(X),
\end{equation}
where $i = 0, \dots, r$ and $t = 1, \ldots, c(i)$. 
Suitably ordered, the relations of the form (\ref{equ1}) and (\ref{equ3}) give rise to a matrix as in Construction 
\ref{constr:ClRel}
\begin{equation}\label{equ:P2}
P
\
:=
\
\left[
\begin{array}{cc|ccc}
-l_{0,1} &l_{0,2} & 0 &\cdots& 0
\\ 
E_{n_0}&E_{n_0} &0& \cdots& 0
\\
\hline
 * & 0 & A(c(1), l_{1,1})& \cdots &0
 \\
 \vdots &\vdots& \vdots & \ddots & \vdots
 \\
* & 0 & 0 & \hdots &  A(c(r), l_{r,1})
\end{array}
\right],
\end{equation}
where we use $c(0)=2$
and the 
$*$ 
indicates that there might be some non-zero entries.
By suitably swapping columns, applying elementary row operations 
and using $l_{0,1} = l_{0,2}$ one achieves
a matrix
$$
P'
\
:=
\
\left[
\begin{array}{c|ccc|c}
-2l_{0,1} & 0&  \multicolumn{2}{c}{\cdots} &0
\\ 
\cline{1-4}
 * & A(c(1), l_{1,1})& \cdots &0 & 0
 \\
 \vdots & \vdots & \ddots & \vdots & \vdots
 \\
 * & 0 & \hdots &  A(c(r), l_{r,1}) & 0
\\
\cline{2-5}
 \multicolumn{1}{c}{E_{n_0}}&0 &\hdots &0 & E_{n_0}
\end{array}
\right].
$$
The rank of $P'$ equals $\sum_{i=0}^r(n_i-1+c(i))$. 
Using
$l_{i,1} = l_{i,2} = l_{i}/2$ for $i = 0,1,2$, we obtain
with Lemma \ref{lemma:matrix}
that the $(n_i-1+c(i))$-th determinantal divisors of $A(c(i),l_{i,1})$ divides 
 $\mathfrak{l}_i/2$ for $i = 1,2$.
Moreover with $c(i) = 4$ and
$l_{i,1} = \ldots = l_{i,4} = l_{i}$ for $i \geq 3$ we obtain 
that the $(n_i-1+c(i))$-th determinantal divisors of $A(c(i),l_{i,1})$ divides 
$\mathfrak{l}_i^3$ for $i \geq 3$.
Thus considering the maximal square submatrices just including one of the first $n_0$ columns,
Laplace expansion with respect to the first row shows that the $\mathrm{rk}(P')$-th determinantal divisor of 
$P'$ divides $\mathfrak{l}_0(\mathfrak{l}_1/2)(\mathfrak{l}_2/2)\mathfrak{l}_3^{3}\ldots\mathfrak{l}_r^{3}$.
Thus $\Cl(X)$ is a factor group of 
$$
\ZZ^{n'}/\mathrm{im}(P^*)\ \cong\ \ZZ^{n'- \mathrm{rk}(P')} \times G,
$$
where $G$ is a finite group of order
$k$ with $k \mid(\mathfrak{l}_0(\mathfrak{l}_1/2)(\mathfrak{l}_2/2)\mathfrak{l}_3^{3}\ldots\mathfrak{l}_r^{3})$. 
Due to Proposition \ref{prop:rkClGroup} 
the rank of the divisor class group equals
$n' - \mathrm{rk}(P')$. Using Proposition \ref{prop:zerodegree} (ii) and the fact that $\Cl(X)^{\mathrm{ctors}}$ does not contain an element of order $2$ due to 
Lemma \ref{lem:comptors}, we conclude that 
$\ZZ/2\ZZ \times \Cl(X)^{\mathrm{ctors}}$ is a subgroup of $G$ of order $\mathfrak{l}_0(\mathfrak{l}_1/2)(\mathfrak{l}_2/2)\mathfrak{l}_3^{3}\ldots\mathfrak{l}_r^{3}$
and thus equality holds. 
\end{proof}

\section{Applications}\label{sec:app}
This section is dedicated to applications of Theorem \ref{thm:divclgr} and is divided into three parts. In the first part we give explicit criteria for the divisor class group of a trinomial variety to be free finitely generated, finite or cyclic and compare our results with existing work in this direction. 
In the second part 
we turn to the iteration of Cox rings for varieties with torus action of complexity one and prove our second main Theorem \ref{thm:link2surf}.
Finally, in Corollary \ref{cor:ClType1} we use our Theorem \ref{thm:divclgr} to compute the divisor class groups of a class of varieties that arises side by side with trinomial varieties as total coordinate spaces of rational varieties with torus action of complexity one.

\begin{corollary}\label{cor:freeab}
Let $X$ be a rational trinomial variety. Then the divisor class group of $X$ is free abelian 
if and only if $X$ is factorial or
after reordering decreasingly we have
$\mathfrak{l}_0 \geq \mathfrak{l}_1 \geq \mathfrak{l}_2 = \ldots  =  \mathfrak{l}_r =1$.
\begin{proof}
Assume that the divisor class group of $X$ is free abelian.
Then either $X$ is factorial and thus $\Cl(X) = \left\{0\right\}$ holds or we may apply Theorem \ref{thm:divclgr}
and conclude $\gcd( \mathfrak{l}_0, \mathfrak{l}_1) > 1$
and
$ \mathfrak{l}_2 = \ldots  =  \mathfrak{l}_r =1$ holds.
The other direction is a direct consequence of Theorem~\ref{thm:divclgr}.
\end{proof}
\end{corollary}

As an application, we consider trinomial
varieties with an isolated singularity;
recall that \cite[Thm. 6.5]{LiSu} gives a complete description of all those with trivial divisor class group.

\begin{corollary}\label{cor:isosing}
Let $X$ be a trinomial variety with an isolated singularity. Then $\dim(X) \leq 5$ holds and we are in one of
the following cases:
\begin{enumerate}
\item 
If $\dim(X) =2$ holds and $X$ is rational then its divisor class group is a torsion group.
\item 
If $\dim(X) =3$ holds then $X$ is rational 
and its divisor class group is free abelian.
\item
If $\dim(X) \geq 4$ holds then $X$ is factorial.
\end{enumerate}
\begin{proof}
Assume $X$ is two-dimensional.
Then $n_i = 1$ holds for
all $i=0, \ldots, r$ and $X$ has an isolated singularity at zero. Thus if $X$ is rational, Theorem \ref{thm:divclgr} implies that its divisor class group is a torsion~group. 

Assume $\dim(X)\geq 3$ holds. Then,
considering the Jacobian of $X$, 
we conclude that $X$ has an isolated singularity at zero 
if and only if 
$X$ is a hypersurface with defining relation
$g= T_0^{l_0}+ T_1^{l_1} + T_2^{l_2}$,
where $1 \leq n_0 \leq n_1\leq n_2 = 2$ and 
$l_{ij} =1$ whenever $n_i =2$, see also \cite{LiSu}.
In particular $\dim(X) \leq 5$ holds and $X$ is rational due to Remark \ref{rem:ratcrit}.
If $n_0 = n_1 = 1$ holds, i.e. $X$ is of dimension three,
we obtain
$\mathfrak{l}_0, \mathfrak{l}_1 \geq \mathfrak{l}_2= 1$.
Applying Corollary \ref{cor:freeab}
we conclude that $X$ is free abelian.
In the case that $n_0 \leq n_1 =n_2 = 2$ holds, i.e. 
$\dim(X) \geq 4$ holds,
we obtain $ \mathfrak{l}_1 = \mathfrak{l}_2 =1$
and thus $X$ is factorial due to Remark~\ref{rem:ratcrit}.
\end{proof}
\end{corollary}

\begin{remark}\label{rem:disc}
As stated in the introduction, Flenner shows in \cite{Fl} 
that rational
three-dimensional 
quasihomogeneous complete intersections
over algebraically closed fields of arbitrary characteristic
with an isolated singularity 
have a free abelian divisor class group.
Corollary \ref{cor:isosing} shows that this is as well true for all trinomial varieties with isolated singularity of dimension at least three. 

Using Corollary \ref{cor:freeab} one can construct 
examples of
affine, rational, trinomial varieties $X$ with free abelian divisor class group having a higher dimensional singular locus: The three-dimensional variety
$$\V(T_{01}^4 + T_{11}^2+ T_{21}^3T_{22}^2)\ \subseteq\  \CC^4$$
has divisor class group $\ZZ$ and a one-dimensional singular locus. Note that not any three-dimensional trinomial variety has a free abelian divisor class group as for instance,
we obtain divisor class group 
$\ZZ \times \ZZ/3\ZZ$ for the hypersurface
$$\V(T_{01}^4 + T_{11}^2+ T_{21}^3T_{22}^3)\ \subseteq\  \CC^4.$$
\end{remark}

\begin{corollary}\label{cor:finclgr}
Let $X$ be a rational trinomial variety. Then the divisor class group of $X$ is finite if and only if 
$X$ is factorial or $n_i = 1$ holds for all $0 \leq i \leq r$. 
\begin{proof}
If $X$ is factorial, then $\Cl(X) = \left\{0 \right\}$ holds. Else we may apply
Theorem \ref{thm:divclgr} and have to evaluate
$0 = \tilde{n} = \sum_{i=0}^{r} ((c_i -1)n_i - c_i +1)$.
This holds if and only if $c_i =1$ holds for all $0 \leq i \leq r$ or $n_i=1$ holds for all $0 \leq i \leq r$. 
The first case is once more the factorial case and the assertion follows.
\end{proof}
\end{corollary}

\begin{corollary}\label{cor:cyclic}
Let $X$ be an adjusted, rational trinomial variety. Then the divisor class group of $X$ is 
a non-trivial finite cyclic group if and only if $r=2$ holds, we have $n_0 = n_1 = n_2 = 1$ and the tuple of exponents $(l_0, l_1, l_2)$ fulfills one of the following conditions:
\begin{enumerate}
\item We have $(l_0, l_1, l_2) = (2x, 2y, z)$ with
pairwise coprime integers $x,y,z$ and $z \geq 3$ is odd.
In this situation we have $\Cl(X) \cong \ZZ/l_2\ZZ$.
\item We have $\gcd(l_0, l_1) = \gcd(l_1, l_2) = \gcd(l_0, l_2) =2$. 
In this situation we have $\Cl(X) \cong \ZZ/((l_0l_1l_2)/4)\ZZ.$
\end{enumerate}
\begin{proof}
As $X$ has non-trivial finite divisor class group we have $n_i = 1$ for all
$0 \leq i \leq r$ due to Corollary \ref{cor:finclgr}. 
Applying Theorem \ref{thm:divclgr} we conclude that the exponents $(l_0, \ldots, l_r)$ are in one of the following two forms: We have
$\gcd(l_0, l_1) =2$ and $\gcd(l_i, l_j) =1$ whenever $j \notin \left\{0,1\right\}$, and there is at most one more exponent $l_i >1$. In this case, as $X$ is adjusted, we conclude $r=2$ and obtain (i).
Else, $\gcd(l_0, l_2) = \gcd(l_1, l_2)= \gcd(l_0,l_2) =2$ holds and we have $l_i =1$ for all $i \geq 3$. As $X$ is adjusted we conclude $r=2$ and obtain Case (ii).
\end{proof}
\end{corollary}

\begin{example}
Applying Corollary \ref{cor:cyclic} we obtain that
the only half-factorial trinomial variety, i.e. a variety with 
divisor class group of order $2$, is the quadric $\V(T_{01}^2 + T_{11}^2 + T_{21}^2)$.
Moreover, we conclude that any rational trinomial hypersurface with 
finite but non-cyclic divisor class group 
is of the form $\V(T_{01}^{l_{01}}+T_{11}^{l_{11}} + T_{21}^{l_{21}})$
with $\gcd(l_{01}, l_{11})> 2$ and $\gcd(l_{01}, l_{21}) = \gcd(l_{11}, l_{21}) = 1$.
\end{example}

\begin{remark}\label{rem:scheja}
Let us compare our results with \cite{La, SchSt, SiSP}, where the divisor class groups of hypersurfaces of the form
$\V(z^n-g)$
with a weighted homogeneous polynomial $g$
of degree relatively prime to $n$ are treated.
In particular the divisor class groups of all trinomial hypersurfaces of the form
$\V(T_{01}^{l_{01}} + T_1^{l_1} + T_2^{l_2})\subseteq \CC^3$
with $\gcd(l_{01}, \mathfrak{l}_1) = 1 =\gcd(l_{01}, \mathfrak{l}_2)$ can be computed with the methods introduced there and are regained as part of our Theorem \ref{thm:divclgr} (i). 
Note that any rational trinomial variety fulfilling
Remark \ref{rem:ratcrit} (iii), as for example the hypersurface
$$\V(T_{01}^2T_{02}^4 + T_{11}^2 + T_{21}^2 T_{22}^6) \ \subseteq\  \CC^5$$
with divisor class group $\ZZ/2\ZZ \times \ZZ^2$,
leaves the framework of \cite{La, SchSt, SiSP}
but can be treated via Theorem \ref{thm:divclgr}.
Further examples are the two hypersurfaces given 
in Remark \ref{rem:disc}.

\end{remark}

Revisiting the proof of Theorem \ref{thm:divclgr} we obtain the following description of the 
$\Cl(X)$-grading of the Cox ring of a rational 
trinomial variety:

\begin{corollary}\label{cor:grad}
Let $X:= \Spec \, R(A,P_0)$ be an adjusted non-factorial rational trinomial variety.
Then the divisor class group grading on the Cox ring
$R(A',P_0')$ is given as
$$\deg(S_k')\ =\ 0, \quad \deg(T_{ij,k})\ =\  Q(e_{ij,k}), \quad \text{with} \quad Q\colon\ZZ^{n'+m}\rightarrow \ZZ^{n'+m}/\im(P^*),$$
where $P$ is one of the following:
\begin{enumerate}
\item
If $c:=\gcd(\mathfrak{l}_0, \mathfrak{l}_1) >1$ and 
$\gcd(\mathfrak{l}_i, \mathfrak{l}_j) = 1$ holds
whenever $j \notin \left\{0, 1\right\}$, then
$P$ is built up as in (\ref{equ:P}).
\item
If $\gcd(\mathfrak{l}_0, \mathfrak{l}_1) = \gcd(\mathfrak{l}_1, \mathfrak{l}_2) = \gcd(\mathfrak{l}_0, \mathfrak{l}_2)=2$ and 
$\gcd(\mathfrak{l}_i, \mathfrak{l}_j) = 1$ holds
whenever ${j \notin \left\{0, 1,2\right\}}$, then
$P$ is built up as in (\ref{equ:P2}).
\end{enumerate}
\end{corollary}

\begin{remark}\label{rem:clhyp}
As a direct consequence of Theorem \ref{thm:divclgr}, 
we  can compute the divisor class groups of 
all hyperplatonic trinomial varieties. We list
the basic platonic tuple (bpt) of $R(A,P_0)$ 
and the divisor class group of ${X:= \Spec \, R(A,P_0)}$
in a table:

\begin{center}
\renewcommand{\arraystretch}{1.8} 
\begin{tabular}{l|l|l}
Case
&
bpt of $R(A,P_0)$
& 
divisor class group
\\
\hline
(i)
&
$(4,3,2)$ 
& 
$\ZZ^{n_1+n_3+\dots+n_r-(r-1)} \times \ZZ/3\ZZ$
\\
\hline
(ii)
&
$(3,3,2)$
& 
$\ZZ^{2\cdot (n_2+\dots+n_r-(r-1))} \times \ZZ/2\ZZ\times \ZZ/2\ZZ$
\\
\hline
(iii)
&
$(x,y,1)$  
& 
$\ZZ^{(\gcd(x,y)-1)\cdot(n_2+\dots+n_r-(r-1))}$
\\
\hline
(iv)
&
$(x,2,2)$ and $2 \nmid x$ 
& 
$\ZZ^{n_0+n_3+\dots+n_r-(r-1)} \times \ZZ/x\ZZ$
\\
\hline
(v)
&
$(x,2,2)$ and $2 \mid x$ 
& 
$\ZZ^{n_0+n_1+n_2+3\cdot(n_3+\dots+n_r-(r-1))} \times \ZZ/x\ZZ$
\end{tabular}.
\end{center}
\end{remark}

With the explicit description of the grading of
the Cox ring of a rational trinomial variety 
given via the matrices $P$ as described in Corollary \ref{cor:grad},
we are able to prove our second main result.

\begin{proof}[Proof of Theorem \ref{thm:link2surf}]
Let $X:= \Spec \, R(A,P_0)$ be hyperplatonic. Then due to Proposition \ref{prop:coxRAP0} the number $m$ of free variables $S_k$ remains the same in the Cox ring $R(A', P_0')$ of $X$. 
Thus, using Corollary \ref{cor:grad}, for our further considerations
we may restrict to the case of hyperplatonic trinomial varieties $\Spec \, R(A,P_0)$ with $m = 0$.
In a first step we show
that for any hyperplatonic ring $R$ with 
basic platonic triple 
$(\mathfrak{l}_0, \mathfrak{l}_1, \mathfrak{l}_2)$,
there exists a good quotient
$\CC^{n} \supseteq \Spec \, R \rightarrow Y(\mathfrak{l}_0, \mathfrak{l}_1, \mathfrak{l}_2)$
with respect to a quasitorus $\TT$.
Setting
$$
\tilde P
\
:=
\
\left[
\begin{array}{ccc}
\frac{1}{\mathfrak{l}_0} l_0
& 
\dots
& 
0
\\
\vdots
& 
\ddots 
& 
\vdots
\\
0 
& 
\dots
&
\frac{1}{\mathfrak{l}_r} l_r
\end{array}
\right],
$$
the map $Q\colon \ZZ^{n} \rightarrow \ZZ^{n}/\im(\tilde P^*)$
defines a grading on $R$, which coarsens the grading given by $P_0$
as in Construction \ref{constr:RAP0}. 
Moreover the Veronese subalgebra $S$ with respect to the degree zero
is generated by the monomials
$T_0^{l_0/\mathfrak{l}_0},\dots, T_r^{l_r/\mathfrak{l}_r}$
and we conclude 
$\Spec \, S \cong Y(\mathfrak{l}_0, \mathfrak{l}_1, \mathfrak{l}_2)$.

Now denote by $R'$ resp. $S'$ the Cox rings of $\Spec \, R$ 
resp. $\Spec \, S$ as given in Proposition \ref{prop::isotropy}
with the grading given by matrices
$P(R)$ resp. $P(S)$ as in Corollary \ref{cor:grad}.
We claim that we obtain the following commutative diagram
$$ 
\xymatrix{
{R'}
&&
{R}
\ar[ll]
\\
{S'}
\ar[u]
&&
{S,}
\ar[u]
\ar[ll]
}
$$
where the upward arrow on the r.h.s. 
is the embedding of a Veronese subalgebra with respect to some free abelian grading group
and the other arrows denote the embeddings of the Veronese subalgebras as defined above.
This proves the assertion as considering the grading 
given by $P(S)$ on $S'$
one directly checks that the 
isomorphism $S' \rightarrow 
\CC[T_0,T_1,T_2]/\bangle{T_0^{\mathfrak{l}_0} +T_1^{\mathfrak{l}_1}+T_2^{\mathfrak{l}_2}}$
deleting the redundant relations is a graded isomorphism with respect to the Cox ring
grading on the latter ring.

To prove our result it is now only 
necessary to show 
that the composition of the 
embeddings $S \rightarrow S'\rightarrow R'$
given by the matrices
$\tilde P$ and $P(S)$
factorizes over the embedding
$R\rightarrow R'$ given by $P(R)$. 
Note that if $P(R)$ is as in (\ref{equ:P})
the grading giving rise
to the composed Veronese embedding
$S \rightarrow S'\rightarrow R'$
can be represented by a matrix
of the same shape and with the same
number of columns as $P(R)$
but replacing the matrices
$A(c(i), l_{i,1})$ by matrices of the form
$$
B(c(i), l_{i,1})
\ :=\ 
\left[
\begin{array}{ccc}
l_{i,1}
& 
\dots
& 
0
\\
\vdots
& 
\ddots 
& 
\vdots
\\
0 
& 
\dots
&
l_{i,1}
\\
l_{i,1}/\mathfrak{l}_{i,1}
&
\dots
&
l_{i,1}/\mathfrak{l}_{i,1} 
\end{array}
\right]
\ \in\  \Mat(k+n_i,k\cdot n_i, \ZZ).
$$
In case that $P(R)$ is as in (\ref{equ:P2}) 
one has to
additionally replace the 
rows $2$ to $n_0+1$ with one row
$(l_{0,1}/\mathfrak{l}_{0,1}, l_{0,1}/\mathfrak{l}_{0,1}, 0,\ldots,0)$.
In particular the row lattice of this matrix
is a sublattice of the row lattice of $P(R)$
and we only have to show that it is a saturated sublattice. 
By the structure of the occurring matrices
this means that the row lattice generated 
by the matrix
$B(c(i),l_{r,1})$
is a saturated sublattice of
the row lattice of the matrix
$A(c(i), l_{i,1})$.
Note that the row lattice of 
$A(c(i), l_{i,1})$ is generated by
the rows of
$$
\left[
\begin{array}{cccc}
l_{i,1}
& 
\dots
& 
0
&
0
\\
\vdots
& 
\ddots 
& 
\vdots
& 
\vdots
\\
0 
& 
\dots
&
l_{i,1}
&
0
\\
E_{n_i} 
&
\dots
&
E_{n_i}
&
E_{n_i}
\end{array}
\right]
\ \in\  \Mat(k+n_i,k\cdot n_i, \ZZ).
$$
In particular the last $n_i$ rows span a saturated sublattice of 
this row lattice. 
As the lattice generated by
$(l_i/\mathfrak{l}_i, \ldots, l_i/\mathfrak{l}_i)$
lies saturated in this sublattice, the assertion follows.
\end{proof}

Besides the trinomial varieties there is another 
class of varieties arising in the context of Cox rings of
varieties with torus action of complexity one:
Due to \cite{HaHe, HaWr} trinomial varieties  
are precisely the total coordinate spaces of
rational varieties $X$ with torus action of complexity one admitting only constant invariant functions. 
Whereas these varieties arise as \emph{Type $2$}
in the description of \cite{HaWr}, the total coordinate spaces of rational varieties $X$ with torus action of complexity one leaving this setting are called \emph{of Type $1$}. They
are given as the common zero locus of relations of the form
$$ 
T_1^{l_1} - T_2^{l_2} - \theta_1,
\quad
T_2^{l_2} - T_3^{l_3} - \theta_2,
\quad
\ldots, \quad
T_{r-1}^{l_{r-1}}  - T_r^{l_r} - \theta_{r-1},
$$
with pairwise different
$\theta_1 =  1, \theta_2, \ldots, \theta_{r-1} \in \CC^*.$
Suitably renumbering and rearranging the trinomials we can achieve $\mathfrak{l}_1 \geq \ldots \geq \mathfrak{l}_r$
and as in the case of Type $2$
we may assume that $n_il_{ij} >1$ holds for all
$0 \leq i \leq n$.
In this situation we call 
$X$ \emph{adjusted} and set
$$ 
c(1)\ :=\ \mathfrak{l}_2, \quad c(2)\ :=\  \mathfrak{l}_1, \quad c(i)\ :=\  \mathfrak{l}_1 \cdot \mathfrak{l}_2 
\quad
\text{for}
\quad
i \geq 3.
$$
Using Theorem \ref{thm:divclgr} we obtain the following explicit description of the divisor class groups of varieties of Type $1$:

\begin{corollary}\label{cor:ClType1}
Let $X$ be a variety of Type $1$, assume $X$ to be adjusted and set $\tilde{n}:= \sum_{i= 1}^r((c(i) -1)n_i - c(i) +1).$
\begin{enumerate}
    \item 
    The divisor class group of $X$ is trivial if and only if one of the following statements hold:
    \begin{enumerate}
        \item 
        We have $\mathfrak{l}_i = 1$ for $1 \leq i \leq r$.
        \item
        We have $X \cong \V(T_1^2 + T_2^2 + 1)$.
    \end{enumerate}
    \item
    The divisor class group of $X$ is free of rank $\tilde{n}$ if and only of one of the following statements hold:
    \begin{enumerate}
        \item 
        We have $\mathfrak{l}_1 >1$ and $\mathfrak{l}_i = 1$ for all $i \geq 2$.
        \item
        We have $\mathfrak{l}_1 = \mathfrak{l}_2 =2$ and $\mathfrak{l}_i =1$ for all $i \geq 3$.
    \end{enumerate}
    \item 
    If we are not in Case (i) or Case (ii), then the divisor class group of $X$ is not finitely generated.
\end{enumerate}
\end{corollary}
\begin{proof}
Let $X$ be a variety of Type $1$ and set $\ell:= \mathrm{lcm}(\mathfrak{l}_1, \ldots, \mathfrak{l}_r)$.
Then due to \cite[Cor. 3.4]{HaWr2} there exist pairwise different $\tilde{\theta}_i \in \CC^*$ such that $X \times \CC^*
\cong \tilde{X}\setminus \V(T_{01})$ holds, where $\tilde{X}$ is the trinomial variety defined by the following relations:
$$ 
T_{01}^\ell + T_1^{l_1} +T_2^{l_2},
\quad
\tilde\theta_1 T_1^{l_1} + T_2^{l_2} +T_3^{l_3},
\
\ldots, \
\tilde\theta_{r-2} T_{r-2}^{l_{r-2}} + T_{r-1}^{l_{r-1}} +T_r^{l_r}.
$$
Denoting by 
$D_{01,1}, \ldots, D_{01, c(0)}$ the prime divisors inside
$V(\tilde{X}; T_{01})$,
the divisor class group $\Cl(X \times \CC^*) \cong \Cl(X)$ 
thus fits into the following exact sequence
\begin{center}
\begin{tikzcd}
\ZZ^{c(0)} \ar[rr,"e_i\mapsto \lbrack D_{01,i}\rbrack"]&&\Cl(\tilde{X})\ar[r]&\Cl(X \times \CC^*)\ar[r]&0.
\end{tikzcd}
\end{center}
In particular $X$ is rational if and only if $\tilde{X}$ is so.
This gives rise to the following three cases:

\vspace{1mm}
\noindent
\emph{Case 1:} We have $\mathfrak{l}_i =1$ for $1 \leq i \leq r$.
In this case $\tilde{X}$ has trivial divisor class group
due to Theorem \ref{thm:divclgr} and thus $X$ has so. 

\vspace{1mm}
\noindent
\emph{Case 2:} We have $\mathfrak{l}_1  > 1$ and $\mathfrak{l}_i  = 1$ for $i \geq 2$.
In this case $\tilde{X}$ is as in Theorem \ref{thm:divclgr} (ii)
and with $\mathfrak{l}_i =1$ for $i \geq 2$ and $n_0 =1$ we conclude
$\Cl(\tilde{X}) \cong \ZZ^{\tilde{n}}$ with $\tilde{n}$ as above. 
Moreover, considering the ring $R(A,P_0)$ corresponding to $\tilde{X}$ we obtain that $T_{01}$ is a $K_0$-prime variable in a $K_0$-factorial ring. Thus due to \cite[Prop. 1.5.3.3]{ArDeHaLa} the vanishing set
$V(\tilde{X}, T_{01})$ is a principal prime divisor.
We conclude $\Cl(\tilde{X}) \cong \Cl(X) \cong \ZZ^{\tilde{n}}$.
We show that no variety with trivial divisor class group arises in this case:
Assume $\tilde{n} =0$. Then $\mathfrak{l}_1 >1$ implies
$n_i =1$  for $i \geq 2$. But as $\mathfrak{l}_i =1$ 
holds for $i \geq 2$ 
this implies $n_i l_{i1} = 1$; in contradiction to $X$ being adjusted. 

\vspace{1mm}
\noindent
\emph{Case 3:} We have $\mathfrak{l}_1= \mathfrak{l}_2  = 2$ and $\mathfrak{l}_i  = 1$ for $i \geq 3$.
In this case $\tilde{X}$ is as in Theorem \ref{thm:divclgr} (iii)
and with $\mathfrak{l}_1 = \mathfrak{l}_2 = 2$ and $\mathfrak{l}_i =1$ for $i \geq 3$ we conclude 
$l_{01} = \ell = 2$ and $\Cl(\tilde{X}) \cong \ZZ/2\ZZ \times \ZZ^{\tilde{n}}$. 
Consider the irreducible components $D_{01,1}$ and $D_{01,2}$ of $\V(\tilde{X}; T_{01})$. Then 
as $n_0 =1$ by construction of $\tilde{X}$,
the class of $D_{01,1}$ generates a subgroup of order 
$2$ in $\Cl(\tilde{X})$ due to Lemma \ref{lem:principal}.
In particular using the above exact sequence we conclude
$\Cl(X) \cong \ZZ^{\tilde{n}}$ as claimed. 
We show that the only variety with trivial divisor class group arising in this case is $\V(T_1^2 + T_2^2 + 1)$.
Assume $\tilde{n} =0$. Then $\mathfrak{l}_1 =2$ implies
$n_i =1$  for $i \geq 2$. As $\mathfrak{l}_i =1$ 
holds for $i \geq 3$ and $X$ is adjusted we obtain $r = 2$. This implies $X \cong \V(T_1^2 + T_2^2 +1)$ as claimed. 
\end{proof}

\bibliographystyle{abbrv}
\bibliography{bib}
\end{document}